\documentclass[12pt]{article}
\usepackage{amsfonts}
\usepackage{amsthm}
\usepackage{amssymb}
\usepackage{latexsym}
\usepackage{amsmath}

\newtheorem{ex}{Example}[section]
\newtheorem{prop}{Proposition}[section]
\newtheorem{twr}{Theorem}[section]

\newcommand{\bex}{\begin{ex}}
\newcommand{\eex}{\end{ex}}
\newcommand{\be}{\begin{equation}}
\newcommand{\ee}{\end{equation}}
\newcommand{\bt}{\begin{twr}}
\newcommand{\et}{\end{twr}}
\newcommand{\bp}{\begin{prop}}
\newcommand{\ep}{\end{prop}}
\newcommand{\R}{\mathbb{R}}
\newcommand{\N}{\mathbb{N}}

\title{Structure of the solution set to differential inclusions with impulses at variable times} 
\author{Agata Grudzka and Sebastian Ruszkowski
\\
\small Faculty of Mathematics and Computer Science,
Nicolaus Copernicus University,
\\
\small Chopina 12/18, 87--100 Toru\'{n}, Poland
\\
\small E-mails: agata33@mat.uni.torun.pl, sebrus@mat.uni.torun.pl}
\date{}
\normalsize

\begin{document}

\maketitle

{\bf Abstract.}
A topological structure of the solution set to differential inclusions with impulses at variable times is investigated. In order to do that an appropriate Banach space is defined. It is shown that the solution set is an $R_{\delta}$-set. Results are new also in the case of~differential equations with impulses at variable times.

{\bf Mathematics Subject Classification (2010):} Primary 34A37; Secondary 34A60, 34K45.

{\bf Keywords and phrases:} solution set, impulsive differential inclusions, variable times, $R_\delta$-set, topological structure.

\section{Introduction} \label{section:introduction}
Impulsive differential equations and inclusions have a~lot of applications in~the~various fields. The moments of impulses can be chosen in various ways: randomly, fixed beforehand, determined by the state of a system. The problems with fixed time of impulses were recently investigated \cite{BeRu, CarRub, DjebaliGor, DjebaliGorOua, GabGru, ObuYao}. The problems with impulses with variable times bring much more difficulties and up to now there were only existence theorems \cite{BeBeOu,BeOu,BeOu2}. Our results develop this research area and show that the solution set is an $R_{\delta}$-set.

There are many motivations to study the structure of solution sets of~differential equations and inclusions. One of them is considering the Poincar\'{e} translation operator and discussing the problem of the existence of periodic solutions \cite{KryPla,CS,HX}.

We have to have the space of functions which contains solutions of~given problem to study the structure of solution set. Obviously standard Banach space with norm $sup$ is insuficient to impulsive problem with the times of~jumps that depend on the state. B-topology on spaces of solutions of impulsive differential inclusions is introduced in \cite{AK}, however, it is only Hausdorff topology. We use this concept to create Banach space that have the same topology on common functions, and is sufficient to the considered problem.

In the Section \ref{section:preliminaries} we describe the problem and recall useful theorems. In~the~main Section \ref{section:main} we present the main results of the paper. The main idea is included in Theorem \ref{twr:dokladniejedenskokINKLUZJA} in which we show that the solution set for the problem with exactly one jump is an $R_{\delta}$-set. In the Theorem \ref{twr:drugieskonWym} we are using the result from previous Theorem proving by induction analogical statement for any fixed number of jumps. We also provide the reader with a transparent example.

\section{Preliminaries} \label{section:preliminaries}

\par The problem we deal with is
\begin{equation}\label{eq:zagadnienieINKLUZJA}
\left\{ \begin{array}{ll}
    \dot{y}(t)\in{F(t,y(t))}, & \hbox{for $t\in[0,a]$, $t\neq{\tau_j(y(t))},$ $j=1,\dots,m,$} \\
    y(0)=y_0, \\
    y(t^{+})=y(t)+I_j(y{(t)}), & \hbox{for $t={\tau_j}(y(t))$, $j=1,\dots,m,$}\\
\end{array}
\right.
\end{equation} 
where $F:[0,a]\times\R^N\multimap\R^N$, $I_j:\R^N\to\R^N,$ $j=1,\dots,m,$ are given impulse functions, $\tau_j\in{C^1(\R^N,\R)}$ with $0<\tau_j(y)<a$, and $t_y=\{t~~|~~t=\tau_k(y(t))\}$. 
The hypersurface $t-\tau_j(y)=0$ is called the $j$-th pulse hypersurface and we denote it by $\Sigma_j$.
If for each $j=1,\dots,m,$ $\tau_j$ is a different constant function, then impulses are in the fixed times.

Our goal is to find the structure of the solution set of the previous problem, but to do that we need a space of functions with $m$ jumps. We introduce considering space as $CJ_m([0,a]):=C([0,a])\times(\R\times\R^N)^m$ with following interpretation:
the element $(\varphi,(l_j,v_j)_{j=1}^m),$ where $l_j\in[0,a]$ we will interpret as the function with $m$ jumps in the times $j_k$ defined as follows:
$$\hat{\varphi}(t):=\left\{
\begin{array}{ll}
\varphi(t), &  0\leq{t}\leq l_{\sigma(1)}, \\
\varphi(t)+\sum\limits_{i=1}^{j}v_{\sigma(i)}, &  l_{\sigma(j)}<t\leq l_{\sigma(j+1)},\\
\varphi(t)+\sum\limits_{i=1}^{m}v_{\sigma(i)}, &  l_{\sigma(m)}<t\leq{a},
\end{array}
\right.$$
where $\sigma$ is a permutation of $\{1, 2, \ldots, m\}$ such that $l_{\sigma(i)}\leq l_{\sigma(i+1)}\mbox{.}$

There is a mutual correspondence between the functions on interval $[0,a]$ with $m$ jumps and the sets  $\{(\varphi,(l_j,v_j)_{j=1}^m)\in CJ_m([0,a]) \; | \; l_j<l_{j+1}\}$, with $\zeta \mapsto (\check{\zeta},(l_j,I_j(\check{\zeta}(l_j)))_{j=1}^m)$, where the function $\check{\zeta}$ is $\zeta$ with reduced jumps, $l_j$ is $j$-th time of jump and the function $I_j$ is an impulse functions.

The space $CJ_m([0,a])$ with the norm
$$\|(\varphi,(l_j,v_j)_{j=1}^m)\|:=\sup_{t\in[0,a]}\|\varphi(t)\|+\sum_{j=1}^{m}(|l_j|+\|v_j\|)$$
is a Banach space.

We will find out that the wanted structure is $R_{\delta}$-type. In order to show that we will use the following well-known theorems:

\begin{twr}[see \cite{Hyman}] \label{twr:Hyman}
Let $X$ be an absolute neighbourhood retract and \mbox{$A\subset{X}$}  be a compact nonempty subset.\par
Then the following statements are equivalent:
\begin{enumerate}
    \item [(a)] $A$ is an $R_\delta$-set,
\item [(b)] for every $\epsilon>0$ the set $A$ is contractible in $O_{\epsilon}(A)=\{x\in{X} ~|~ dist(x,A)<{\epsilon}\}$,
\item [(c)] $A$ is an intersection of a decreasing sequence $\{A_n\}$ of compact contractible spaces,
\item [(d)] $A$ is an intersection of a decreasing sequence $\{A_n\}$ of closed contractible spaces, such that $\beta(A_n)\to 0$, where $\beta$ is the Hausdorff measure of noncompactness.
\end{enumerate} 
\end{twr}
 \begin{twr} {\em(Convergence theorem) (see \cite{AE}\label{twr:convergenceinkluzje})}
Let $E$ and $E'$ be Banach spaces, $(T,\Omega,\mu)$ be a measurable space, and the multivalued map $F\! : \! T\times E\multimap E'$ has closed and convex values and for a.e. $t\in T$ the map $F(t,\cdot)\! : \! E\multimap E'$ is uhc. Let $(u_n\! :\! T\to E)$ be a sequence of functions such that $u_n\to u$ in $L^p(T,E)$ and let sequence $(w_n)\subset L^p(T,E')$, $1\leq p<\infty$ be such that $w_n\rightharpoonup{w}$ in $L^p(T,E')$. If for a.e. $t\in T$ and for arbitrary $\varepsilon>0,$ there exists $N\in\N$ such that
$$w_n(t)\in \mbox{cl}\,\mbox{conv}\, B(F(t,B(u_n(t),\varepsilon)),\varepsilon)$$ for $n>N$, then $w(t)\in F(t,u(t))$ for a.e. $t\in T$.
\end{twr}

We recall Arzela- Ascoli Theorem:
\begin{twr}\label{twr:Arzela-Ascoliklasyczne}
If the family $\mathcal{F}\subset{C([a,b],\R^N)}$ of continuous functions is equicontinuous and uniformly bounded, then  there exists a subsequence that converges uniformly. 
\end{twr}

A piecewise absolutely continuous function $y \! :\! [0,a]\to{\R^N}$ is a solution of the problem with impulses (\ref{eq:zagadnienieINKLUZJA}) if:
\begin{enumerate}
\item[(a)] $y(0)=y_0,$
\item[(b)] there exists a function $f\in{L^1([0,a],\R^N)}$ such that $f(t)\in{F(t,y(t))}$ for a.e. $t\in[0,a]$ and $y(t)=y_0+\sum_{j=1}^{m}I_j(y(t_j))+\int_{0}^{t}f(s)ds,$ where $t_j=\tau_{j}(y(t_j)),$ 
\item[(c)] the function $y$ is left continuous at $t={\tau}_j(y(t))\in[0,a]$ and the limit $y(t^+)$ exists and  $y(t^+)=y(t)+I_j(y(t))$ for $t={\tau}_j(y(t)),$ $j=1,\dots,m.$
\end{enumerate} 

\section{The structure of the solution set}\label{section:main}
We assume the following conditions on multivalued perturbation $F:[0,a]\times{\R^N}\multimap{\R^N}$:

   \begin{enumerate}
    \item[$(F0)$] $F$ has compact and convex values,
    \item[$(F1)$] $F(\cdot,y):[0,a]\multimap{\R^N}$ has a measurable selection for every $y\in\R^N$,
    \item[$(F2)$] is almost uniformly with respect to $t$ H-usc, i.e. for every $y\in\R^N$ and $\varepsilon>0$ there exists $\delta>0$ such that for a.e. $t\in[0,a]$ and for all $x\in\R^N$ if $\|y-x\|<\delta$, then $\sup_{\varphi\in{F(t,x)}} d(\varphi,F(t,y))<\varepsilon,$
    \item[$(F3)$] has a sublinear growth, i.e., there exists $\alpha\in{L^1([0,a])}$ such that
 $$\sup_{\varphi\in{F(t,y)}}\|\varphi\|\leq\alpha(t)(1+\|y\|) \textrm{ for a.e. }
 t\in[0,a] \mbox{ and } y\in{\R^N}\mbox{.}$$

Moreover, we assume the following hypotheses about impulse functions:
    \item[$(H1)_m$] $I_j\in{C(\R^N,\R^N)}$, $j=1,\dots,m$,
    \item[$(H2)_m$] $\tau_j\in{C^1(\R^N,\R)},$ $j=1,\dots,m$, \\
for $j=1,\dots,m-1$ we have:
$$0<\tau_j(y)<\tau_{j+1}(y)<a,$$ 
$$\tau_j(y+I_j(y))\leq\tau_j(y)<\tau_{j+1}(y+I_j(y)),$$
$$\tau_m(y+I_m(y))\leq\tau_{m}(y),$$ and there exists a constant $M\geq{0}$ such that $\|{\tau_j}'(y)\|\leq{M}$ for all $y\in{\R^N}$, $j=1,\dots,m$,  
    \item[$(H3)_m$] there exists a constant $p>0$ such that for a.e. $t\in[0,a]$  $$\sup_{\varphi\in{F(t,y)}}{\tau}_j'(y)\cdot \varphi-1\leq{-p}<0 \mbox{ for all } y\in{\R^N}, j=1,\dots,m \mbox{.}$$ 
 \end{enumerate}

Note taht if ${\tau}'_j(y)=0$ for $j=1,\dots,m$, then the problem is reduced to a problem with a fixed impulse time.

The assumption $(F2)$ and compact values of the multivalued map $F$ implies that $F(t,\cdot)$ is usc.

\begin{twr}\label{twr:dokladniejedenskokINKLUZJA} Let the assumptions $(F0)-(F3)$ and $(H1)_1-(H3)_1$ hold. Then every solution of the problem (\ref{eq:zagadnienieINKLUZJA}), where $m=1$, meets $\Sigma_1$ exactly once and the solution set $S$ of this problem is an $R_\delta$-set in the space $CJ_1([0,a])$. 
\end{twr}

\begin{proof}
To simplify notation we write $I$ and $\tau$ instead of $I_1$ and $\tau_1$. 
We will proceed in several steps.

 \vspace{0,3cm}
\underline{Step 1.}
A Lipschitz selection and the uniqueness of jump.

For each $n$ let $\{B(y,r_n(y))\}_{y\in{\R^N}}$ be an open covering (open balls, such that $r_n(y)\leq{\frac{1}{n}}\to{0}$) of the space $\R^N,$ such that for every $x\in B(y,r_n(y))$ we have

\be\label{eq:szacowanie}
\sup_{\varphi\in{F(t,x)}}d(\varphi,F(t,y))<\frac{1}{n}.\ee
There exists locally finite open point-star refinement $\mathcal{U}_n=\{U_{n,s}\}_{s\in{S}}$ of the cover
$\{B(y,r_n(y))\}_{y\in{\R^N}},$ i.e. for every  $y\in{\R^N}$ there exists $x_{y,n}\in\R^N$ such that $\mbox{st}(y,\mathcal{U}{_n})\subset{B(x_{y,n},r_n(x_{y,n}))}.$

We can choose it in a way that $\mathcal{U}_{n+1}$ is a~refinement of the cover $\mathcal{U}_{n}$. Let $\{\lambda_s\}_{s\in{S}}$ be a locally Lipschitz partition of unity subordinated to the cover $\mathcal{U}_n$ i.e. for every $s\in{S}$ the function $\lambda_s:[0,a]\to\R^N$ satisfies the locally Lipschitz condition.
For every $y_s\in\R^N$, $s\in{S}$ let a~function $q_s$ be a~measurable selection of ${F}(\cdot,y_s).$ We define 
the function $g_n:[0,a]\times{\R^N}\to{\R^N}$ in the following way
$$g_n(t,y):=\sum_{s\in{S}}\lambda_s(y)\cdot{q_s(t)}.$$

The set $S(y):=\{s\in{S}~|~{\lambda}_s(y)\neq{0}\}$ is finite. If $\lambda_s(y)>0,$ so $s\in{S(y)},$ then $y\in\mbox{supp\,}\lambda_s\subset U_{n,s}\subset{\mbox{st\,}(y,\mathcal{U}_n)}.$ There exists ${x_{y,n}}\in\R^N$ such that $\mbox{st\,}(y,\mathcal{U}{_n})\subset B(x_{y,n},r_n(x_{y,n})).$ We know that $y_s\in{\mbox{st\,}(y,\mathcal{U}_n)}$.
We get
\begin{align*}\nonumber
g_n(t,y)&=\sum_{s\in{S}}\lambda_s(y)\cdot{q_s(t)}\in\mbox{conv\,}F(t,\mbox{st\,}(y,\mathcal{U}{_n}))\\
&\subset G_n(t,y):=\mbox{cl conv\,}F(t,\mbox{st\,}(y,\mathcal{U}{_n}))\mbox{.}\\
\end{align*}
Moreover,
$${G_n(t,y)}\subset\mbox{cl conv\,}F(t,B(x_{y,n},r_n(x_{y,n}))),$$
so from the inequality $(\ref{eq:szacowanie})$  we have 
\be\label{eq:GnsiedziwOtoczceF}
G_n(t,y)\subset \mbox{ cl\,}O_{\frac{1}{n}}(F(t,x_{y,n})).\ee
We have
$$F(t,y)\subset\bigcap_{n\geq{1}}G_n(t,y).$$

From  usc (the map $F$ has compact values) we get that for every $y\in\R^N$ and for every $\varepsilon>0$ there exists $\delta>0$ such that
$$F(t,B(y,\delta))\subset{O_{\varepsilon}}(F(t,y)).$$

We have
\begin{align*}\nonumber
\bigcap_{n\geq{1}}G_n(t,y)&\subset \bigcap_{n\geq{1}}\mbox{ cl\,}O_{\frac{1}{n}}(F(t,x_{y,n}))\\
&\subset\bigcap_{n\geq{1}}\mbox{ cl\,}O_{\frac{1}{n}}\left(F\left(t,B(y,2r_n(x_{y,n}))\right)\right)\\
&\subset\bigcap_{n\geq{1}}\mbox{ cl\,}O_{\frac{1}{n}+\varepsilon}\left(F(t,y)\right)=F(t,y)\mbox{.}\\
\end{align*}
We obtain that
$$F(t,y)=\bigcap_{n\geq{1}}G_n(t,y).$$
We have
$$G_{n+1}(t,y)\subset G_n(t,y).$$

Let us introduce the Nemitski\u{i} (substitution) operator $P_{G_n} : CJ_1([0,a])\multimap L^1(J,\R^N)$
as follows
$$P_{G_n}(y):=\{\phi\in{L^1([0,a],\R^N)}~|~ \phi(t)\in{G_n(t,y(t))} \textrm{ for a.e. } t\in[0,a] \}.$$
Let $S_n$ denote the set of solutions of the problem
\begin{equation}\label{eq:zagadnieniezGnvariabletimes}
\left\{ \begin{array}{ll}
    \dot{y}(t)\in{G_n(t,y(t))}, & \hbox{for $t\in[0,a]$, $t\neq{\tau(y(t))},$} \\
    y(0)=y_0, \\
    y(t^{+})=y(t)+I(y{(t)}), & \hbox{for $t={\tau}(y(t))$. }\\
\end{array}
\right.
\end{equation} 
It is obvious that the sets $S_n$ are nonempty, because the problem
\begin{equation}\label{eq:zagadnieniegnvariabletimes}
\left\{ \begin{array}{ll}
    \dot{y}(t)=g_n(t,y(t)), & \hbox{for $t\in[0,a]$, $t\neq{\tau(y(t))},$} \\
    y(0)=y_0, \\
    y(t^{+})=y(t)+I(y{(t)}), & \hbox{for $t={\tau}(y(t))$, }\\
\end{array}
\right.
\end{equation}
for every $n\in\N$ has exactly one solution.

 \vspace{0,3cm}
\underline{Step 1a.}

We denote by $t^j_{y_n}$ the time of $j$-th jump for the function $y_n$ and if the function $y_n$ has less that $j$ jumps we take $t^j_{y_n}=a$.   
Let $y_n$ be an arbitrary solution of the system (\ref{eq:zagadnieniezGnvariabletimes}) for $0\leq t\leq t^2_{y_n}$. For $t\leq{t^1_{y_n}}$ we get the following form of the solution:
$$y_n(t)=y_0+\int_{0}^{t}\phi_n(s)ds,$$
where $\phi_n\in P_{G_n}(y_n)$. 

There exists selection $f_n$ (not necessarily measurable) of the multivalued map $F(\cdot,x_{y,n}(\cdot))$ such that for a.e.  $t$ we get $\|\phi_n(t)-f_n(t,x_{y,n}(t))\| \leq \frac{1}{n}$. 
We have
$$\|\phi_n(t)\|\leq \frac{1}{n}+\|f_n(t,x_{y,n}(t))\|.$$
From the assumption $(F3)$ we obtain
$$\|f_n(t,x_{y,n}(t))\|\leq \alpha(t)\left(1+\|x_{y,n}(t)\|\right)\leq \alpha(t)\left(1+\|y(t)\|+\frac{1}{n}\right).$$
So
$$
\|y_n(t)\|\leq \|y_0\|+\int_{0}^{t}\left(\alpha(s)\left(1+\|y_n(s)\|+\frac{1}{n}\right)+\frac{1}{n}\right)ds\mbox{.}
$$
From Gronwall inequality, we have:
\be\label{eq:wspOgr}\sup_{t\in[0,t^1_{y_{n}}]}\|y_n(t)\|\leq{\left(\|y_0\|+
\int_{0}^{a}2\alpha(s)ds+\frac{a}{n}\right)e^{\int_0^a\alpha(s)ds}}:=K.\ee
By continuity of the function $I$ there exists a constant $c>0$ such that $\|I(y_n(t^1_{y_n}))\|\leq c$ for all $n$.
Next for $t^1_{y_n}<t\leq{t^2_{y_n}}$ we obtain 
\begin{align*}\nonumber
\|y_n(t)\|&\leq\|y_0\|+\|I(y_n(t^1_{y_n}))\|+\int_{0}^{t}\|\phi_n(s)\|ds\\
&\leq\|y_0\|+c+\int_{0}^{t}\left(\alpha(s)\left(1+\frac{1}{n}+\|{y_n(s)}\|\right)+\frac{1}{n}\right)ds\mbox{.}\\
\end{align*}
Again, from Gronwall inequality we get:
\be\label{eq:wspOgrPosokou}\sup_{t\in[0,t^2_{y_n}]}\|y_n(t)\|\leq{Ce^{\int_0^t\alpha(s)ds}}<{Ce^{\int_0^a\alpha(s)ds}}=:\bar{K},\ee
where $C:=\|y_0\|+c+\int_{0}^{a}2\alpha(s)ds+\frac{a}{n}.$\\
If the solution $y_n$ does not have jumps, then of course $\sup_{t\in[0,a]}\|y_n(t)\|\leq K.$ \\

\vspace{0,3cm}
\underline{Step 1b.}

Let $\bar{y}$ be a fixed function with values in $\mbox{cl\,}B(0,\bar{K}).$ Let $\phi_n\in{G_n(t,\bar{y}(t))}$, where $t$ is such that $(H3)_1$ is satisfied. We denote $\bar{y}(t)=:y\mbox{.}$ 
From the assumptions $(H2)_1$ and $(H3)_1$ for some $v\in\mbox{cl }{B(0,1)}$ we get:
\begin{align*}\nonumber
{\tau}'(y)\cdot{\phi}_n-1
&=\left({{\tau}'(x_{y,n})+\tau}'(y)-{\tau}'(x_{y,n})\right)\cdot{\phi}_n-1\\
&={\tau}'(x_{y,n})\cdot\left(\bar{\varphi}+\frac{1}{n}v\right)-1+\left({\tau}'(y)-{\tau}'(x_{y,n})\right)\cdot{\phi}_n\\
&={\tau}'(x_{y,n})\cdot\bar{\varphi}-1+\frac{1}{n}{\tau}'(x_{y,n})\cdot{v}+({\tau}'(y)-{\tau}'(x_{y,n}))-{\phi}_n,\\
\end{align*}
where $\bar{\varphi}\in{F(t,x_{y,n})}.$
Hence
\begin{align*}\nonumber
{\tau}'(y)\cdot{\phi}_n-1&\leq{-p+\frac{\|{\tau}'(x_{y,n})\|}{n}}+\left({\tau}'(y)-{\tau}'(x_{y,n})\right)\cdot{\phi}_n\\
&\leq{-p+\frac{M}{n}}+\|{\tau}'(y)-{\tau}'(x_{y,n})\|~~\|{\phi}_n\|.\\
\end{align*}

The function ${\tau}'$ is continuous on the compact set $\mbox{ cl\,}B(0,\bar{K}+1)$, therefore it is uniformly continuous. Hence, for $y\in\mbox{ cl\,}B(0,\bar{K})$ ${\tau}'(y)-{\tau}'(x_{y,n})\to{0}$ (we have $\|x_{y,n}\|\leq{\bar{K}}+\frac{1}{n}$). Moreover, the set $\{{\phi}_n\}$ is bounded (from the sublinear growth of $F$). We have $\|{\tau}'(y)-{\tau}'(x_{y,n}))\|~~\|{\phi}_n\|\to{0}.$  
We take $N_0\in{\N}$ such that for every $n\geq{N_0}$ we have $-p+\frac{M}{n}+\|{\tau}'(y)-{\tau}'(x_{y,n})\|~~\|{\phi}_n\|<\frac{-p}{2}.$  There exists a constant $p'>0$ such that 
$${\tau}'(y)\cdot \phi_n-1<-p'\mbox{.}$$

\vspace{0,3cm}
\underline{Step 1c.}

Let us fix $y_n$, where $n>N_0,$ the solution of the problem (\ref{eq:zagadnieniezGnvariabletimes}).

We define the function $w_n:[0,a]\to\R$ by:
$$w_n(t):=\tau(y_n(t))-t.$$
The function $w_n$ has value $0$ in any time, in which the function $y_n$ has a~jump.
By the condition $(H2)_1$ we get $w_n(0)=\tau(y_0)>{0}$ and $w_n(a)=\tau(y_n(a))-a<a-a=0$.
If $w_n(t)\neq{0}$ on $[0,a]$, then there would not be any impulse effect, therefore there would not be any jump time, so $w_n$ would be continuous, which would contradict with the earlier inequalities. Hence every solution of~the~problem (\ref{eq:zagadnieniezGnvariabletimes}) has at least one jump.

Suppose that $0<t^1_{y_n}<a$ is the first time in which the solution $y_n$ hits the hypersurface $\Sigma_1$. Then
$$w_n(t^1_{y_n})=0\;\mbox{ and } w_n(t)>0,\; \mbox{ for } t\in[0,t^1_{y_n}).$$
By  assumption $(H2)_1$ we get that
$$w_n({t^1_{y_n}}^+)=\tau(y_n({t^1_{y_n}}^+))-{t^1_{y_n}}= \tau\left(y_n(t^1_{y_n})+I(y_n(t^1_{y_n}))\right)-t^1_{y_n}\leq{0}\mbox{.}$$
For a.e. $t\geq{t^1_{y_n}}$ we have  
$$w'_n(t)={\tau}'(y_n(t))\cdot{y_n'(t)}-1={\tau}'(y_n(t))\cdot{\phi}_n(t)-1<-p'<0,$$
where $\phi_n\in{P_{G_n}(y_n)}$.
The function $w_n$ in $[t^1_{y_n},a]$ is decreasing, hence $y_n$ hits the hypersurface $\Sigma_1$ exactly once and the time of this jump we denote $t_{y_n}\mbox{.}$  

\vspace{0,3cm}
\underline{Step 2.}

Now we will show that each sequence $(y_n)$, where $y_n\in{S_n},$ has a~convergent subsequence to the solution $\tilde{y}$ of the problem (\ref{eq:zagadnienieINKLUZJA}).\\

There exists exactly one jump, so from the previous estimations we have

$$\|y_n(t)\|\leq\bar{K}\mbox{.}$$
Consequently, the values of solutions of the problem (\ref{eq:zagadnienieINKLUZJA}) are contained in a ball $\mbox{cl\,}B(0,\bar{K})$, which is convex, so in particular we know that function $g_n|_{[0,a]\times\mbox{cl}\,B(0,\bar{K})}$ has integrable Lipschitz constant $\Lambda$.

For $t<t'\leq t_{y_{n}}$ we have
\begin{equation}\begin{split} \label{eq:JednakowaCG}
\|y_n(t)-y_n(t')\|&=\left\|\int_{t}^{t'}\phi_n(s)ds\right\|\leq \frac{1}{n}|t-t'|+\int_{t}^{t'}\alpha(s)\left(1+\frac{1}{n}+\|y_n(s)\|\right)ds\\
&\leq |t-t'|+(2+K)\int_{t}^{t'}\alpha(s)ds,
\end{split}\end{equation}
and for $t_{y_{n}}<t<t'$ we have
\begin{equation}\begin{split} \label{eq:JednakowaCG2}
\|y_n(t)-y_n(t')\|&=\left\|\int_{t}^{t'}\phi_n(s)ds\right\|\leq \frac{1}{n}|t-t'|+\int_{t}^{t'}\alpha(s)\left(1+\frac{1}{n}+\|y_n(s)\|\right)ds\\
&\leq |t-t'|+(2+\bar{K})\int_{t}^{t'}\alpha(s)ds.
\end{split}\end{equation}

\underline{Step 2.a} 

Let us consider convergence to the time $t_*$, where $t_*$ the limit of~a~convergent subsequence $(t_{y_{n_k}})$ of the sequence $(t_{y_n})$ (with $y_n\in{S_n}$), which exists due to the compactness of $[0,a]$. 
For every $\varepsilon>0$ there exists $N_0$ such that for $n_k>N_0$ we have $t_*-\varepsilon<t_{y_{n_k}}$.

Note that
$$\sup_{t\in[0,t_*-\varepsilon-\delta]}\int_{t}^{t+\delta}\alpha(s)ds\to{0}$$
with $\delta\to 0\mbox{.}$  

From inequality (\ref{eq:JednakowaCG}) we obtain that for all $\xi>0$ and $0<t<t_*-\varepsilon$ there exists $\delta>0$ such that for all $n_k>N_0$ and $t<t'<t+\delta\leq t_*-\varepsilon$ we have $\|y_{n_k}(t)-y_{n_k}(t')\|\leq |t'-t|+(2+K)\int_{t}^{t'}\alpha(s)ds<\xi$.
Therefore family $\{y_{n_k}\}_{n_k>N_0}$ is equicontinuous and by the inequality (\ref{eq:wspOgr}) uniformly bounded. 
By Arzela-Ascoli Theorem \ref{twr:Arzela-Ascoliklasyczne} (possibly going to the subsequences) we can assume that $y_{n_{k}}\to{y_{\varepsilon}}$ on $[0,t_*-\varepsilon]$, where $y_{\varepsilon}$ is continuous function. This can be done in such a way that for any ${\varepsilon}_1,{\varepsilon}_2>0$ such that ${\varepsilon}_1>{\varepsilon}_2$ functions $y_{{\varepsilon}_1}$, $y_{{\varepsilon}_2}$ fulfil condition $y_{\varepsilon_2}|_{[0,t_*-{\varepsilon}_1]}=y_{{\varepsilon}_1}\mbox{.}$
For ${\varepsilon}\searrow{0}$ we obtain an extension of the function $y_{\varepsilon}$, ie the function $y_*\! :\! [0,t_*) \to \R^N$, where $y_{n_k}$ converges pointwise to $y_{*}\mbox{.}$

Moreover
\begin{equation}
\label{eq:calkOgranphi_n} \|\phi_{n_k}(t)\|\leq{\alpha(t)\left(1+\|x_{y_{n_k},n_k}(t)\|\right)}+\frac{1}{n_k}\leq \alpha{(t)}(2+\bar{K})+1. \end{equation}
We know that: 
\begin{itemize}
 \item ${\phi}_{n_k}(t)\in \mbox{cl conv}\; F\left(t,\mbox{st\,}(y_{n_k}(t),\mathcal{U}_{n_k})\right)
\subset{\mbox{cl}}\,O_{\frac{1}{n_k}}\left(F\left(t,B(y_{n_k}(t),\frac{1}{n_k})\right)\right)$, by inclusion (\ref{eq:GnsiedziwOtoczceF}),
 \item $y_{n_k}(t)\to y_{\varepsilon}(t)$ a.e. on $[0,t_*-\varepsilon]$,
 \item ${\phi}_{n_k}\in{L^1([0,t_*-\varepsilon],\R^N)}$,
 \item by estimation (\ref{eq:calkOgranphi_n}) and weak compactness of the closed ball we get ${\phi}_{n_{k_l}}\rightharpoonup{\phi}$ on $[0, t_*-\varepsilon]$. 
\end{itemize}
Thus by Theorem \ref{twr:convergenceinkluzje} we obtain ${\phi}(t)\in F(t,y_{\varepsilon}(t))$ for a.e. $t\in[0,t_*-\varepsilon]$.
By analogy, we conclude that ${\phi}(t)\in F(t,y_*(t))$ a.e. on $[0,t_*)$.
By weak convergence
${\phi}_{n_{k_l}}\rightharpoonup{\phi}$ on $[0, t_*-\varepsilon]$, for $\Psi(\phi_n):=\int_{0}^{t}\phi_n(s)ds$ we have
\begin{align*}
\int_{0}^{t}\phi(s)ds&=\Psi(\phi)=\lim_{k\to\infty}\Psi(\phi_{n_k})=\lim_{k\to\infty}\int_{0}^{t}{\phi}_{n_k}(s)ds\\
&=\lim_{k\to\infty}{y_{n_k}(t)}-y_0=y_*(t)-y_0\mbox{.}\end{align*}
For an increasing sequence $(s_n)$ convergent to $t_*$ with $n<n'$ we obtain
$$\|y_*(s_n)-y_*(s_{n'})\|=\left\|\int_{s_n}^{s_{n'}}\phi(s)ds\right\|\leq \int_{s_n}^{t_*}(\alpha(s)(2+K)+1)ds\mbox{.}$$
We have convergence of the right hand side of the inequality to $0$ with $n\to \infty$,
consequently $(y_*(s_n))$ is Cauchy sequence. It is convergent ($\R^N$ is complete) and we denote its limit $y_*(s_n)\to y_1$. We define $y_*(t_*):=y_1$ and we obtain continuous extension $y_*$ on $[0,t_*]$.
We will show, that $\tau(y_*(t_*))-t_*=0$, which means that $t_*$ is time of jump for $y_*$.

Let $\epsilon>0$. The function $\alpha(s)(2+\bar{K})+1$ is integrable, so we can choose $t_{\epsilon}<t_*$ so that  $$\int_{t_{\epsilon}}^{t_*}2[\alpha(s)(2+\bar{K})+1]ds<\frac{\epsilon}{2}\mbox{.}$$
There exists $K_0\in\N$ such that for $k>K_0$ we have $\|y_*(t_{\epsilon})-y_{n_k}(t_{\epsilon})\|<\frac{\epsilon}{2}$.
We can estimate
\begin{align*}
\|y_{n_k}(t_{y_{n_k}})-y_*(t_{*})\|&\leq \left\|y_{n_k}(t_{\epsilon})+\int_{t_{\epsilon}}^{t_{y_{n_k}}}\phi_{n_k}(s)ds-y_*(t_{\epsilon})-\int_{t_{\epsilon}}^{t_*}\phi(s)ds\right\|\\  
\leq&\|y_{n_k}(t_{\epsilon})-y_*(t_{\epsilon})\|+\int_{t_{\epsilon}}^{t_{y_{n_k}}}\|\phi_{n_k}(s)\|ds+\int_{t_{\epsilon}}^{t_*}\|\phi(s)\|ds\\
\leq&\|y_{n_k}(t_{\epsilon})-y_*(t_{\epsilon})\|+\int_{t_{\epsilon}}^{t_{y_{n_k}}}\left(\alpha(s)(2+\bar{K})+1\right)ds\\
+&\int_{t_{\epsilon}}^{t_*}\left(\alpha(s)(2+\bar{K})+1\right)ds\\
\leq& \frac{\epsilon}{2}+\int_{t_{\epsilon}}^{t_*}2(\alpha(s)(2+\bar{K})+1)ds\leq \epsilon\mbox{.}
\end{align*}
For $t_{y_{n_k}}>t_*$ we get
\begin{align*}\|y_*(t_*)-y_{n_k}(t_{y_{n_k}})\|\leq \|y_*(t_*)-y_*(t_*-\varepsilon)\|&+\|y_*(t_*-\varepsilon)-y_{n_k}(t_*-\varepsilon)\|\\
&+\|y_{n_k}(t_*-\varepsilon)-y_{n_k}(t_{y_{n_k}})\|\mbox{,}
\end{align*} 
but it is easy to see that:
\begin{align*}
\|y_{n_k}(t_*-\varepsilon)-y_{n_k}(t_{y_{n_k}})\|&=\left\|\int_{t_*-\varepsilon}^{t_{y_{n_k}}}\phi_{n_k}(s)ds\right\|
\leq\int_{t_*-\varepsilon}^{t_{y_{n_k}}}\|\phi_{n_k}(s)\|ds\\
&\leq  \int_{t_*-\varepsilon}^{t_{y_{n_k}}} \big(\alpha(s)(2+K)+1\big)ds\\
&= (2+K)\int_{t_*-\varepsilon}^{t_{y_{n_k}}} \alpha(s)ds +t_{y_{n_k}}-t_*+\varepsilon \mbox{,}
\end{align*}
so
\begin{align*}
\|y_*(t_*)-y_{n_k}(t_{y_{n_k}})\|&\leq \|y_*(t_*)-y_*(t_*-\varepsilon)\|+\|y_*(t_*-\varepsilon)-y_{n_k}(t_*-\varepsilon)\|\\
&+(2+K)\int_{t_*-\varepsilon}^{t_{y_{n_k}}} \alpha(s)ds +t_{y_{n_k}}-t_*+\varepsilon\\
&\to_{k\to \infty}
\|y_*(t_*)-y_*(t_*-\varepsilon)\|+(2+K)\int_{t_*-\varepsilon}^{t_*} \alpha(s)ds +\varepsilon\mbox{.}
\end{align*}
From the arbirary $\epsilon$ and $\varepsilon$ we get
$\|y_*(t_*)-y_{n_k}(t_{y_{n_k}})\|\to 0$.

Summarising, we have that 
$y_{n_k}(t_{y_{n_k}})\to y_*(t_*)$.
By the continuity of $\tau$ we obtain $\tau(y_*(t_*))-t_*=\lim_{{n_k}\to\infty}(\tau(y_{n_k}(t_{y_{n_k}}))-t_{y_{n_k}})=0$ which means that $t_*$ is the time of jump for $y_*$.

 \vspace{0,3cm}
\underline{Step 2.b}

We will make the similar reasoning with the part of segment $[0,a]$ after jump.

Form inequalities (\ref{eq:wspOgrPosokou}) and (\ref{eq:JednakowaCG2}) we conclude that the family $\{y_{n_k}\}$ is equicontinuous and equibounded on $[t_*+\varepsilon,a]$. 
Therefore, from Arzeli-Ascolego \ref{twr:Arzela-Ascoliklasyczne} theorem (passing to a subsequence if it is
needed) we can assume that $y_{n_k}\to{y^{\varepsilon}}$, where $y^{\varepsilon}$ is a continuous function and we extend
it to a continuous function $y^*\! :\! (t_*,a] \to \R^N$ with $y_{n_k}$ convergent pointwise to $y^*$.

We know that:
\begin{itemize}
 \item ${\phi}_{n_k}(t)\in \mbox{cl conv}\; F\left(t,\mbox{st\,}(y_{n_k}(t),\mathcal{U}_{n_k})\right)
\subset{\mbox{cl}}\,O_{\frac{1}{n_k}}\left(F\left(t,B(y_{n_k}(t),\frac{1}{n_k})\right)\right)$.
 \item $y_{n_k}(t)\to y^{\varepsilon}(t)$ a.e. on $[t_*+\varepsilon,a]$,
 \item ${\phi}_{n_k}\in{L^1([t_*+\varepsilon,a],\R^N)}$,
 \item by estimation (\ref{eq:calkOgranphi_n}) and the weak compactness of closed ball we get ${\phi}_{n_{k_l}}\rightharpoonup{\phi}$ on $[t_*+\varepsilon,a]$. 
\end{itemize} 
Again, by theorem \ref{twr:convergenceinkluzje} we obtain ${\phi}(t)\in F(t,y^{\varepsilon}(t)),$ for a.e. $t\in{[t_*+\varepsilon,a]}$, so we get an information that ${\phi}(t)\in F(t,y^*(t))$ a.e. on $(t_*,a]$.
By weak convergence
${\phi}_{n_{k_l}}\rightharpoonup{\phi}$ on $[t_*+\varepsilon,a]$ for $\Psi(\phi_n):=\int_{t}^{a}\phi_n(s)ds$ we have
\begin{align*}\nonumber
\int_{t}^{a}\phi(s)ds&=\Psi(\phi)=\lim_{k\to\infty}\Psi(\phi_{n_k})=\lim_{k\to\infty}\int_{t}^{a}{\phi}_{n_k}(s)ds\\
&=\lim_{k\to\infty}{y_{n_k}(a)}-y_*(t)=y_*(a)-y_*(t)\mbox{.}\\
\end{align*}
For decreasing sequence $(s_n)$ convergent to $t_*$ we get for $n<n'$:
$$\|y^*(s_n)-y^*(s_{n'})\|=\left\|\int_{s_{n'}}^{s_n}\phi(s)ds\right\|\leq \int^{s_n}_{t_*}\alpha(s)(1+\bar{K})ds\mbox{.}$$
By analogy to Step 2a, we get that $(y^*(s_n))$ is the Cauchy sequence, which is convergent in Banach space, so $y^*(s_n)\to y_2$ for some $y_2\in\R^N$.

By continuity of $I$ we have that $I(y_{n_k}(t_{y_{n_k}}))\to I(y_*(t_*))$, therefore $y_2=y_*(t_*)+I(y_*(t_*))$.

We can define a function $\tilde{y} \! : \! [0,a] \to \R^N$ by concatenation $y_*$ on $[0,t_*]$ with the function $y^*$ on $(t_*,a]$.
Obviously, for $t\leq t_*$ the function $\tilde{y}$ is the solution of the problem~(\ref{eq:zagadnienieINKLUZJA}). For $t>t_*$ we have:
$$\tilde{y}(t)=\tilde{y}(t_*)+I(\tilde{y}(t_*))+\int_{t_*}^t \phi(s)ds=y_0+\int_0^{t_*}\phi(s)ds+ I(\tilde{y}(t_*))+\int_{t_*}^t \phi(s)ds\mbox{,}$$
where $\phi\in P_{G_n}(\tilde{y})$, so $\tilde{y}$ is the solution of the problem~(\ref{eq:zagadnienieINKLUZJA}) for all $t\in[0,a]$, hence $\tilde{y}\in{S}$. The function $\tilde{y}$ is
the limit of the sequence $(y_{n_k})$ in the space $CJ_1([0,a])$.
 \vspace{0,3cm}

\underline{Step 3.}
We will show, for every $n\in\N$,  the contractibility of the set $\mbox{cl}\, S_n$.

Fix $n$ such that we can define $p'$ (see Step 1.) and take $\bar{y}\in\mbox{cl}\,S_{n}$. We divide the interval $[0,1]$ into two halfs.

Let $r\in[0,\frac{1}{2}]$. We consider the problem
\begin{equation}\left\{
\begin{array}{ll}
    \dot{y}(t)=g_n(t,y(t)), & \mbox{for } t\in[a-2r(a-t_{\bar{y}}),a] \mbox{, } t\neq{\tau(y(t)}) \mbox{,}  \\
    y(t)=\bar{y}(t), & \mbox{for } t\in[0,a-2r(a-t_{\bar{y}})] \mbox{,}\\
    y(t^+)=y(t)+I({y}(t)), & \mbox{for } t=\tau(y(t)) {.}\\
\end{array}
\right.\label{eq:impulDlaselekcjigSKwym}\end{equation} 
In the previous problem we denote by $g_n$ selection of the map $G_n.$
There exists exactly one solution of this problem; we denote it by $y^{2}_{\bar{y},r}.$ Then
$y^2_{\bar{y},r}\in\mbox{cl}\,S_n.$\\
Next for $r\in(\frac{1}{2},1]$ we consider the problem
\begin{equation}\left\{
\begin{array}{ll}
    \dot{y}(t)=g_n(t,y(t))\mbox{,} & \mbox{for } t\in[t^{\bar{y},r},a] \mbox{, } t\ne\tau(y(t)) \mbox{,} \\
    y(t)=\bar{y}(t)\mbox{,} & \mbox{for } t\in[0,t^{\bar{y},r}] \mbox{,}\\
y(t^+)=y(t)+I({y}(t))\mbox{,} & \mbox{for } t=\tau(y(t)) \mbox{.} \\
\end{array}
\right.\label{eq:mimpulDlaselekcjigSKwym}\end{equation}
where $t^{\bar{y},r}:=t_{\bar{y}}-2(r-\frac{1}{2})t_{\bar{y}}\mbox{.}$
There exists exactly one solution of this problem, denoted by $y^{1}_{\bar{y},r},$ which also belongs to $\mbox{cl}\,S_n.$\\
Finally we consider the following function $h:[0,1]\times{\mbox{cl\,}S_n}\rightarrow \mbox{cl\,}S_n$ given by:
\begin{equation}h(r,\bar{y}):=\left\{
\begin{array}{ll}
y^2_{\bar{y},r}, & \hbox{ $r\in[0,\frac{1}{2}]$},\\
y^1_{\bar{y},r}, & \hbox{ $r\in(\frac{1}{2},1]$}.\\
\end{array}
\right.\end{equation}

\vspace{0,5cm}
Now, we will show that the function h is continuous.

\vspace{0,5cm}
Due to the continuous dependence of solutions on initial conditions \cite{Hale} we know that the function $h$ is continuous on $[0,\frac{1}{2})\times{\mbox{cl}\,S_n}$ and left continuous on $\{\frac{1}{2}\}\times\mbox{cl}\,S_n$. 

Let $((r_k,\bar{y}_k))_k$  be a sequence convergent to $(\frac{1}{2}^+,\bar{y})$.

We know that if $t_{\bar{y}}<t_{h(r_k,\bar{y}_k)}$ then $\tau(h(r_k,\bar{y}_k)(t_{\bar{y}}))-t_{\bar{y}}>0$ and we have
$$ t_{\bar{y}}-\tau(h(r_k,\bar{y}_k)(t_{\bar{y}}))= \int_{t_{\bar{y}}}^{t_{h(r_k,\bar{y}_k)}}(\tau(h(r_k,\bar{y}_k)(\cdot))-\cdot)'(\theta)d\theta < -p'(t_{h(r_k,\bar{y}_k)}-t_{\bar{y}}),$$
so
$$ t_{h(r_k,\bar{y}_k)}<t_{\bar{y}}+(\tau(h(r_k,\bar{y}_k)(t_{\bar{y}}))-t_{\bar{y}})/p'\mbox{,}$$
but for $k$ such that $t_{\bar{y}}\leq t_{h(r_k,\bar{y}_k)}$ one sees that
\begin{align*}
\Bigg\|\bar{y}(t^{\bar{y}_k,r_k})&+\int_{t^{\bar{y}_k,r_k}}^{t_{\bar{y}}}\bar{\phi}_n(s)ds- \left(y^1_{\bar{y}_k,r_k}(t^{\bar{y}_k,r_k})
+\int_{t^{\bar{y}_k,r_k}}^{t_{\bar{y}}}g_n(s,y^1_{\bar{y}_k,r_k}(s))ds\right)\Bigg\|\\
&=\left\|\int_{t^{\bar{y}_k,r_k}}^{t_{\bar{y}}}\bar{\phi}_n(s)-g_n(s,x^1_{\bar{y}_k,r_k}(s))ds\right\|\\
&\leq\int_{t^{\bar{y}_k,r_k}}^{t_{\bar{y}}}\|\bar{\phi}_n(s)\|+\|g_n(s,y^1_{\bar{y}_k,r_k}(s))\|ds\\
&\leq\int_{t^{\bar{y}_k,r_k}}^{t_{\bar{y}}}2\left(\alpha(s)(1+K)+\frac{1}{n}\right)ds\mbox{.}
\end{align*}
Hence, if $k\to\infty$ we get 
$$ h(r_k,\bar{y}_k)(t_{\bar{y}})\to \bar{y}(t_{\bar{y}})\mbox{,}$$
so
$$ \tau( h(r_k,\bar{y}_k)(t_{\bar{y}}))-t_{\bar{y}}\to 0\mbox{.}$$
For every $k$ we have
$$t^{\bar{y}_k,r_k}<t_{h(r_k,\bar{y}_k)}\leq t_{\bar{y}}+{1{\hskip -2.5 pt}\hbox{l}}_{t_{\bar{y}}<t_{h(r_k,\bar{y}_k)}}(\tau(h(r_k,\bar{y}_k)(t_{\bar{y}}))-t_{\bar{y}})/p'$$
and
$$t^{\bar{y}_k,r_k}\rightarrow t_{\bar{y}}\leftarrow t_{\bar{y}}+{1{\hskip -2.5 pt}\hbox{l}}_{t_{\bar{y}}<t_{h(r_k,\bar{y}_k)}}(\tau(h(r_k,\bar{y}_k)(t_{\bar{y}}))-t_{\bar{y}})/p'\mbox{,}$$
therefore by squeeze theorem
\be\label{eq:t}
t_{h(r_k,\bar{y}_k)}\rightarrow t_{\bar{y}}\mbox{.}
\ee

\vspace{0,5cm}
Let us fix k for a while. By $1{\hskip -2.5 pt}\hbox{l}$ we denote the function:
\begin{equation}\nonumber
1{\hskip -2.5 pt}\hbox{l}_{t>t_0}(y(t))=
\left\{ \begin{array}{ll}
    y(t), & \hbox{for $t>t_0,$} \\
    0, & \hbox{for $t\leq{t_0}.$} \\
   \end{array}
\right.
\end{equation}
We define function $\varrho_{\bar{y}_k,r_k}\! :\! [0,a]\to\R^N$ 

\begin{align*}
\varrho_{\bar{y}_k,r_k}(t):=&y^1_{\bar{y}_k,r_k}(t)- {1{\hskip -2.5 pt}\hbox{l}}_{t>t_{y^1_{\bar{y}_k,r_k}}}(I(y^1_{\bar{y}_k,r_k}(t_{y^1_{\bar{y}_k,r_k}})))\\
-&(y^2_{\bar{y},1/2}(t)-{1{\hskip -2.5 pt}\hbox{l}}_{t>t_{\bar{y}}} (I(y^2_{\bar{y},1/2}(t_{\bar{y}}))),
\end{align*}
which is function of differences between $y^1_{\bar{y}_k,r_k}\mbox{,}$ and $y^2_{\bar{y},1/2}$ with deleted changes caused
by jumps. It is easy to see that $t_{\bar{y}}=t_{y^2_{\bar{y},1/2}}\mbox{.}$
For $t^{\bar{y}_k,r_k}\leq t\leq t_{\bar{y}}$ we get that:
\begin{align*}
\|\varrho_{\bar{y}_k,r_k}(t)\|&\leq \|\bar{y}-\bar{y}_k\|+\int_{t^{\bar{y}_k,r_k}}^t \|g_n(s,{y^1_{\bar{y}_k,r_k}}(s))-\bar{\phi}_n(s)\|ds\\
&\leq\|\bar{y}-\bar{y}_k\|+\int_{t^{\bar{y}_k,r_k}}^{t}\left(\|\bar{\phi}_n(s)\|+\|g_n(s,y^1_{\bar{y}_k,r_k}(s))\|\right)ds\\
&\leq\|\bar{y}-\bar{y}_k\|+\int_{t^{\bar{y}_k,r_k}}^{t}2\left(\alpha(s)(1+\bar{K})+\frac{1}{n}\right)ds=:z_{\bar{y}_k,r_k}(t)
\end{align*}
and for $t> t_{\bar{y}}$ we have
\begin{align*}
\|\varrho_{\bar{y}_k,r_k}(t)\|&\leq z_{\bar{y}_k,r_k}(t_{\bar{y}})+ \int_{t_{\bar{y}}}^{t}\|g_n(s,{y^1_{\bar{y}_k,r_k}}(s))-g_n(s,{y^2_{\bar{y},1/2}}(s))\|ds 
\\
\leq z_{\bar{y}_k,r_k}(t_{\bar{y}})&+\int_{t_{\bar{y}}}^{t} \Lambda(s)\|{y^1_{\bar{y}_k,r_k}}(s)-{y^2_{\bar{y},1/2}}(s)\|ds
\\
=z_{\bar{y}_k,r_k}(t_{\bar{y}})&+\int_{t_{\bar{y}}}^{t}\Lambda(s)\|{\check{y}^1_{\bar{y}_k,r_k}}(s)
+{1{\hskip -2.5 pt}\hbox{l}}_{s>t_{y^1_{\bar{y}_k,r_k}}}I(y^1_{\bar{y}_k,r_k}(t_{y^1_{\bar{y}_k,r_k}}))
-{\check{y}^2_{\bar{y},1/2}}(s)
\\
&-I({y}^2_{\bar{y},1/2}(t_{\bar{y}}))\|ds
\\
\leq z_{\bar{y}_k,r_k}(t_{\bar{y}})&+ \left|\int_{t_{\bar{y}}}^{t_{y^1_{\bar{y}_k,r_k}}}\Lambda(s)\|I({y}^2_{\bar{y},1/2}(t_{\bar{y}})\|ds\right|
\\
&+\int_{\max\{t_{\bar{y}},t_{y^1_{\bar{y}_k,r_k}}\}}^t \Lambda(s)\|I(y^1_{\bar{y}_k,r_k}(t_{y^1_{\bar{y}_k,r_k}}))-I({y}^2_{\bar{y},1/2}(t_{\bar{y}}))\|ds
\\
&+\int_{t_{\bar{y}}}^{t}\Lambda(s)\|{\check{y}^1_{\bar{y}_k,r_k}}(s)-{\check{y}^2_{\bar{y},1/2}}(s)\|ds,
\end{align*}
hence by the Gronwall inequality we obtain
\begin{align*}
\|\varrho_{\bar{y}_k,r_k}(t)\|&\leq \Bigg(z_{\bar{y}_k,r_k}(t_{\bar{y}})+\left|\int_{t_{\bar{y}}}^{t_{y^1_{\bar{y}_k,r_k}}} \Lambda(s)\|I({y}^2_{\bar{y},1/2}(t_{\bar{y}}))\|ds\right|\\
&+\int_{\max\{t_{\bar{y}},t_{y^1_{\bar{y}_k,r_k}}\}}^t \Lambda(s)\|I(y^1_{\bar{y}_k,r_k}(t_{y^1_{\bar{y}_k,r_k}}))- I({y}^2_{\bar{y},1/2}(t_{\bar{y}}))\|ds\Bigg)\\
&*\exp{\int_{t_{\bar{y}}}^{t}\Lambda(s)ds}=:z_{\bar{y}_k,r_k}(t)\mbox{.} 
\end{align*}
By previous convergences and continuity of I we obtain

\begin{equation}\begin{split}\label{eq:jumps}
\Big\| &I(h(r_k,\bar{y}_k)(t_{h(r_k,\bar{y}_k)})) -I(h(1/2,\bar{y})(t_{h(1/2,\bar{y})}))\Big\|\to 0
\end{split}\end{equation}
and
\be\label{eq:supy}
\sup_{t\in[0,a]}\|\varrho_{\bar{y}_k,r_k}(t)\|\leq z_{\bar{y}_k,r_k}(a)\to 0.
\ee

Summing up, by (\ref{eq:supy}), (\ref{eq:t}) and (\ref{eq:jumps}), if $(r_k,\bar{y}_k)$ converges to $(\frac{1}{2}^+,\bar{y}),$ then $y^1_{\bar{y}_k,r_k}$ converges to $y^2_{\bar{y},\frac{1}{2}}$ in norm in the space $CJ_1([0,a])$.

The function h, as continuous on
$[0,1]\times{\mbox{cl}\,S_n},$ is a homotopy. By definition of the function $h$ we have
$h(0,\bar{y})=\bar{y}$ and $h(1,\bar{y})=y^1_{\bar{y},1}$, so
${\mbox{cl}\,S_n}$ is a~contractible set.
 \vspace{0,3cm}

\underline {Step 4.}

We will show that properties needed to theorem \ref{twr:Hyman} are fulfilled.

The sets $\mbox{cl}\,S_n$ are contractible in the power of Step 3.

If $x\in \bigcap_{n\in\N}\; \mbox{cl}\,S_n$, then $x\in \mbox{cl}\,S_n$ for every $n$. Therefore, there exists sequence $(d_n)\subset\R_+$ converges to~$0$ such that $B(x,d_n)$ (in $CJ_1([0,a])$) contains $y_n$ ($y_n\in{S_n}$). Hence $y_n\to x$ in the space $CJ_1([0,a])$. Moreover, we know that subsequence $y_{n_k}$ converges to a solution of problem (\ref{eq:zagadnienieINKLUZJA}), where $m=1$, so $x\in S$. We get
$$S\subset\bigcap_{n\in\N}\, S_n\subset\bigcap_{n\in\N}\,\mbox{cl}\, S_n\subset S\mbox{,}$$
so $S=\bigcap_{n\in\N}\,\mbox{cl}\, S_n$.

We will show that $\sup\{d(z,S)| z\in{S_n}\}\to_{n\to\infty}{0}$. Assume that there exists $\varepsilon>0$ and a sequence $(y_n)$ such that $y_n\in{S_n}$ and $d(y_n,S)\geq\varepsilon$. From the Step 3 we know taht this sequence has subsequence $(y_{n_k})$ such that $y_{n_k}\to{\tilde{y}}\in{S},$
so $d(y_{n_k},S)\to{0}$. It is contrary to the choice of the sequence $(y_n),$ hence $\sup\{d(z,S)|
z\in{S_n}\}\to{0}$. Therefore
$\sup\{d(z,S)| z\in{\mbox{cl}\,{S_n}}\}\to{0}$. 

We get
$S_n\subset S + B(0,p_n)$, where $p_n:=\sup_{z\in S_n}d(z,S)\to 0$ with $n\to\infty$.

Compactness of $S$ implies
$$\beta(\mbox{cl}\, S_n)=\beta(S_n)\leq \beta(S)+p_n=p_n,$$
so $\beta(\mbox{cl}\, S_n)\to 0$.

Summing up, we can use theorem \ref{twr:Hyman}, which implies that the set $S$ is an $R_{\delta}$-set.
\end{proof}

We will use this theorem to prove more general case.

\begin{twr} \label{twr:drugieskonWym} Let the assumptions $(F0),$-$(F3)$ and $(H1)_m,$-$(H3)_m$ hold. Then every
solution of the problem (\ref{eq:zagadnienieINKLUZJA}) for every $j=1,\dots,m$ meets $\Sigma_j$ exactly once and the solution set $S$ of this problem is an $R_\delta$-set in the space $CJ_m([0,a]).$
\end{twr}
\begin{proof}
We will show that we can divide the interval $[0, a]$ into $m$ disjoint parts and any
of them will have exactly one jump effect. Then we will be able to use the reasoning
of theorem \ref{twr:dokladniejedenskokINKLUZJA} on every such part, which will end the proof. By analogy to Step 1. in
the proof of Theorem \ref{twr:dokladniejedenskokINKLUZJA} we define a~multivalued map $G_n:[0,a]\times\R^N\multimap\R^N$ and consider the following problem 
\begin{equation}\label{eq:zagadnieniezGnvariabletimesk}
\left\{ \begin{array}{ll}
    \dot{y}(t)\in{G_n(t,y(t))}, & \hbox{for $t\in[0,a]$, $t\neq{\tau_j(y(t))},$ $j=1,\dots,m,$} \\
    y(0)=y_0, \\
    y(t^{+})=y(t)+I_j(y{(t)}), & \hbox{for $t={\tau}_j(y(t)),$ $j=1,\dots,m$. }\\
\end{array}
\right.
\end{equation} 

We denote by $t^j_{y_n}$ the time of $j$-th jump for the function \mbox{$y_n:[0,a]\to\R^N$}. If the function $y_n$ has less than $j$ jumps we take $t^j_{y_n}:=a$.
Let $y_n$ be an~arbitrary solution of the problem (\ref{eq:zagadnieniezGnvariabletimesk}) for $0\leq{t}\leq{t^2_{y_n}}\mbox{.}$
By analogy to theorem \ref{twr:dokladniejedenskokINKLUZJA} in Step 1a. we show that there exists a constant $\bar{K}$ such that
$$\sup_{t\in[0,t^2_{y_n}]}\|y_n(t)\|\leq{\bar{K}}\mbox{.}$$
Next we will proceed similary to the proof of theorem \ref{twr:dokladniejedenskokINKLUZJA} (Step 1b), we show that there exists a constant $p'>0$ such that 
$${\tau}'_j(y)\cdot\phi_n-1<-p'$$
for all $j=1,\dots,m,$ and for enough big $n,$
where $\bar{y}$ is fixed function with values in $\mbox{cl}\,B(0,\bar{K})\mbox{,}$ $t$ is such that the assumption $(H3)_m$ is satysfied, $\bar{y}(t)=y,$ $\phi_n\in{G_n(t,\bar{y}(t))}.$ 
We define the function $w_{n,j}:[0,a]\to\R$ in the following way:
$$w_{n,j}(t):=\tau_j(y_n(t))-t,~~~~ j=1,\dots,m\mbox{.}$$
Let us fix a solution $y_n$ of problem (\ref{eq:zagadnieniezGnvariabletimesk}).

Now we will prove by induction that the $j$-th time of
jump is zero of~the~function $w_{n,j}$.\\
{\underline {Basis.}}  By the assumption $(H2)_m$ we have:\\
a) $w_{n,j}(0)=\tau_j(y_0)>0,$\\
b) $w_{n,j}(a)={\tau}_j(y_n(a))-a<0,$\\
c) $w_{n,j}(t)=\tau_j(y_n(t))-t<\tau_{j+1}(y_n(t))-t=w_{n,j+1}(t)$ for all $t\in[0,a]\mbox{.}$

If there are no impulses then $w_{n,j}(t)\neq{0}$ on $[0,a],$ but by the
definition of $w_{n,j}$ for every $j$ the function $w_{n,j}$ is continuous, which contradicts with a) and b). Hence there is at least one jump time.

Let $0<t^1_{y_n}<a$ be the first time of jump for the solution $y_n\mbox{.}$ Then
$$w_{n,j}(t^1_{y_n})=0,\quad w_{n,j}(t)>0,\; \mbox{ dla } t\in[0,t^1_{y_n})\mbox{.}$$
By the assumption $(H2)_m$ we get that:
$$w_{n,j}({t^1_{y_n}}^+)=\tau_j(y_n({t^1_{y_n}}^+))-{t^1_{y_n}}= \tau_j(y_n(t^1_{y_n})+I_j(y_n(t^1_{y_n})))-t^1_{y_n}\leq{0}\mbox{.}$$
For a.e. $t\geq{t^1_{y_n}}$ we have  
$$w'_{n,j}(t)={\tau}'_j(y_n(t))\cdot{y_n'(t)}-1={\tau}_j'(y_n(t))\cdot{\phi}_n(t)-1<-p'<0\mbox{.}$$
Hence $w_{n,j}$ is decreasing function at $[t^1_{y_n},a]$, so $y_n$ for every $j$ meets $\Sigma_j$ exactly once. 
By c) we know that the time $t^1_{y_n}$ of the first jump is zero of the function $w_{n,1}$.\\
{\underline {Inductive step.}}
We assume that the $j$-th time of jump is zero of $w_{n,j},$ $j<m$.\\
Denote by $t_{y_n}^j$ the time for which we have $t=\tau_j(y_n(t))$.
We know that $w_{n,l}({t^j_{y_n}}^+)=\tau_l(y_n(t^j_{y_n}))-{t^j_{y_n}}>\tau_j(y_n(t^j_{y_n}))-{t^j_{y_n}}=0$ for $j>1\mbox{,}$ $l>j$.
We consider interval $J:=(t^j_{y_n},a]$. By analogy to Step 1 in proof of theorem \ref{twr:dokladniejedenskokINKLUZJA} we have $w_{n,i}(t)\neq{0},$ $t\in{J},$ $i\leq{j}$, as long as there is
no jump caused by $w_{n,l}$ for $l>j$.
There have to be at least one jump after $t^j_{y_n}$ and we denote it by~$\tilde{t}$. By c) there are no jumps before the jump caused by $w_{n,j+1}$, so $\tilde{t}=t^{j+1}_{y_n}$.

Since both the basis and the inductive step have been proved, it has been proved
by mathematical induction that there are exactly $m$ jumps one for each $\Sigma_j$, $j=1,\dots,m$.

Next we will proceed similary to the proof of theorem \ref{twr:dokladniejedenskokINKLUZJA}.\\
Let $(t^1_*, t^2_*, \dots, t^m_*)$ be the limit of the sequence $((t^1_{y_n}, t^2_{y_n}, \dots, t^m_{y_n}))$.\\
We can show that on $[0, t^1_*]$ there exists the subsequence of $(y_{n_k})$ such that $I_1(y_{n_{k_l}}(t_{y_{n_{k_l}}}))\to{I_1(y_*(t^1_*))}$ and $\tau_1(y_*(t_*^1))=t^1_*\mbox{,}$ so
$$\tau_2(y_*((t_*^1)^+))=\tau_2(y_*(t_*^1)+I_1(y_*(t_*^1)))>\tau_1(y_*(t_*^1))=t_*^1\mbox{.}$$
By the assumptions $(H1)_m$ and $(H2)_m$ there exists $\varepsilon_1>0$ such that there is only
one jump time of $y_*$ on $[0,t^1_*+2\varepsilon_1]$ and we get that for all sufficiently big $l$ there is only
one jump time of $y_{n_{k_l}}$ on $[0,t^1_*+\varepsilon_1]$. By analogy we find subsequence and $\varepsilon_2>0$ such that there is exactly one jump on $[t_*^1+\varepsilon_1, t_*^2+\varepsilon_2]$, and so on. On every such interval we
proceed with reasoning from theorem \ref{twr:dokladniejedenskokINKLUZJA}.
\end{proof}

The following example shows the case of inclusion with exactly one jump, but can easly be rearranged to a multijump case.

\bex
There are two trust funds with interest rates (dependent on time and amount of money) $\alpha(t,y_1)\in A(t,y_1)$ and $\beta(t,y_2)\in B(t,y_2)$ respectively where $A$ and $B$ are (multivalued) investment plans.
It is available to transfer money once (in or out) without loosing interest.

We decided to start both trust funds with the same amount of money and transfer money from worse deposit to better one after clarifying which one is better.
The amount of money that we wish to transfer would be proportional to difference in incomes.

 This situation can be represented in the form of following differential inclusion:
$$
\left\{ \begin{array}{ll}
    \dot{(y_1,y_2)}(t)\in F(t,(y_1,y_2)(t)),&\hbox{for $t\in[0,1]$, $t\neq{\tau((y_1,y_2)(t))}$,} \\
    (y_1,y_2)(0)=(y_0,y_0), \\
    (y_1,y_2)(t^{+})=(y_1,y_2)(t)+I((y_1,y_2)(t)),&\hbox{for $t={\tau}((y_1,y_2)(t))$}\\
\end{array}
\right.
$$

with
$$ F(t,y_1,y_2)=A(t,y_1)\times B(t,y_2),$$
where $A$ and $B$ fulfill assumptions $(F1)$ and $(F2)$ (for example continuous) and have interval values with $A(t,y), \; B(t,y) \in [-6y,6y]$
$$I(y_1,y_2)=\left\{
\begin{array}{ll}
(-y_1,y_1), &  0\leq (1+\frac{1}{\rho}) y_1<y_2, \\
(y_2,-y_2), &  0\leq (1+\frac{1}{\rho}) y_2<y_1, \\
(\rho (y_1-y_2), \rho (y_2-y_1)), &  0\leq y_1\leq (1+\frac{1}{\rho}) y_2\leq (1+\frac{1}{\rho})^2 y_1, \\
(0,0), & y_1<0 \mbox{ or } y_2<0
\end{array}
\right.$$
and
$$\tau(y_1,y_2):=\mbox{arccot}(y_1+y_2)/\pi.$$
\vspace{0.1cm}

For any $\alpha\in A(t,y)$ we have that
$$\alpha \geq -3(1/2+2y^2)=(\pi(p-1))(1/2+2y^2)$$
where $1>p=1-3/\pi>0$. We have alnalogous inequality for $\beta$ therefore we get
\begin{align*}
\sup_{\varphi\in{F(t,y_1,y_2)}}{\tau}_j'(y_1,y_2)\cdot \varphi &= \sup_{\alpha\in{A(t,y_1)}}\sup_{\beta\in{B(t,y_2)}} \frac{\alpha+\beta}{-\pi(1+(y_1+y_2)^2)}\\
&\leq (1-p)\frac{1+2y_1^2+2y_2^2}{1+(y_1+y_2)^2}\leq{1-p}<1
\end{align*}

All assumptions of Theorem \ref{twr:dokladniejedenskokINKLUZJA} are satisfied, therefore the solution set of~this~problem is an~$R_\delta$-set.
\eex


\end{document}